\documentclass[11pt]{article}
\usepackage[english]{babel}
\usepackage{amssymb,amsmath,amsthm}
\newtheorem{theorem}{Theorem}[section]
\newtheorem{lemma}[theorem]{Lemma}
\newtheorem{proposition}[theorem]{Proposition}
\newtheorem{corollary}[theorem]{Corollary}

\newtheorem{question}[theorem]{Question}

{\theoremstyle{definition}}
{\theoremstyle{definition}\newtheorem{definition}[theorem]{Definition}}
{\theoremstyle{definition}}

\newtheorem*{thmgs}{Theorem GS}
\newtheorem*{corgs}{Corollary GS}
\newtheorem*{corgs1}{Corollary GS1}

\numberwithin{equation}{section}

\def\N{{\mathbb N}}
\def\Z{{\mathbb Z}}

\def\K{{\mathbb K}}

\def\epsilon{\varepsilon}
\def\kappa{\varkappa}
\def\phi{\varphi}
\def\leq{\leqslant}
\def\geq{\geqslant}

\def\dim{{\rm dim}\,}
\def\ssub#1#2{#1_{{}_{{\scriptstyle #2}}}}
\def\dimk{{\ssub{\dim}{\K}\,}}

\def\spann{\hbox{\tt span}\,}

\def\deg{\hbox{\tt deg}\,}

\title{Optimal $5$-step nilpotent quadratic algebras}

\author{Natalia Iyudu, Stanislav Shkarin}

\date{}

\begin{document}

\maketitle

\begin{abstract}By the Golod--Shafarevich Theorem, an associative
algebra $R$ given by $n$ generators and $d<\frac{n^2}{3}$
homogeneous quadratic relations is not 5-step nilpotent. We prove
that this estimate is optimal. Namely, we show that for every
positive integer $n$, there is an algebra $R$ given by $n$
generators and $\bigl\lceil\frac{n^2}{3}\bigr\rceil$ homogeneous
quadratic relations such that $R$ is 5-step nilpotent.
\end{abstract}



\section{Introduction \label{s1}}\rm

In the series of papers \cite{na, ns, ns1}
we study the question on whether the famous Golod-Shafarevich estimate, which
gives a lower bound for the Hilbert series of a (noncommutative) algebra, is attained.
This question was considered by Anick in his 1983 paper, "Generic algebras and
CW-complexes" \cite{ani1}, in connection with the study of rationality of Poincare series associated to loop spaces. Here instead of looking at the Hilbert series as a whole,  we concentrate on the conditions of vanishing of its certain component, that is on the condition of the nilpotency of given index. It turns out that in the case when the lower estimate of the number of relations necessary for the vanishing of a given component is rational, we can apply our technique to prove that this estimate is attained.


For  a wider information on related problems about Golod-Shafarevich algebras see
\cite{Zelm, Sm1,Sm2,Sm3, Ersh}, mentioned above papers \cite{na, ns, ns1}, and references therein.

Throughout this paper $\K$ is an arbitrary field, $\Z_+$ is the
set of non-negative integers and $\N$ is the set of positive
integers. For a real number $t$, $\lceil t\rceil$ is the smallest integer
$\geq t$. For a set $X$, $\K\langle X\rangle$ stands for the free
algebra over $\K$ generated by $X$ equipped with the natural degree grading, where
each $x\in X$ has degree $1$. The symbol $\langle X\rangle$ stands for the set
of all monomials (=words) in the alphabet $X$.

We deal with {\it quadratic algebras}, that is, algebras $R$ given
by homogeneous relations of degree 2. Such algebras inherit the described above grading of  $\K\langle X\rangle$ and are degree
graded. If $X$ is finite, one can consider the Hilbert series of $R$:
\begin{equation*}
H_R(t)=\sum_{q=0}^\infty (\dimk R_q)\,t^q,
\end{equation*}
where $R_q$ is the space of homogeneous elements of $R$ of degree $q$. See \cite{popo,ufnar,vers} for more information on
quadratic algebras.

\begin{definition}\label{kstep} A quadratic algebra $R$ is called
{\it $k$-step nilpotent} if $R_k=\{0\}$.
\end{definition}

Note that if $R$ is $k$-step nilpotent, then $R_m=\{0\}$ for $m\geq
k$. Thus $R$ is $k$-step nilpotent if and only if $H_R$ is a polynomial
of degree $<k$, provided $X$ is finite.

On the Hilbert series we consider a componentwise ordering:
for two power series $a(t)$ and $b(t)$ with real coefficients we
write $a(t)\geq b(t)$ if $a_j\geq b_j$ for all $j\in\Z_+$.
Recall the following remarkable estimate of Golod and Shafarevich \cite{gosh}.
Although their result applies to any algebra defined by homogeneous relations,
we state it here in the case of quadratic algebras.

\begin{thmgs}Let, $n\in\N$, $0\leq d\leq
n^2$ and $R$ be a quadratic $\K$-algebra with $n$ generators and $d$
relations. Then $H_R(t)(1-nt+dt^2)\geq 1$.
\end{thmgs}

In \cite{popo} it is observed that Theorem~GS implies the following result.

\begin{corgs}If $R$ is a quadratic $\K$-algebra given by $n$
generators and $d<\phi_k n^2$ relations, then $\dim R_k>0$, where
$\frac1{\phi_k}=4\cos^{2}\bigl(\frac{\pi}{k+1}\bigr)$.
\end{corgs}

Anick \cite{ani1,ani2} proved that the estimate in
Corollary~GS is optimal for $k=3$. Namely, $\phi_3=\frac12$ and
Anick has shown that for every $n\in\N$, there is a quadratic
$\K$-algebra $R$ given by $n$ generators and
$\bigl\lceil\frac{n^2}{2}\bigr\rceil$ relations
such that $R$ is $3$-step nilpotent. The authors  proved in \cite{ns1}
that the estimate from Corollary~GS is asymptotically optimal for
every $k\geq 3$. So far, the only numbers $k$ for which the
optimality was known were $k=2$ (trivial) and $k=3$ due to Anick. In
\cite{na}, the authors have developed a technique of constructing
quadratic algebras with small dimensions of the graded components. In the
present paper this technique is enhanced and used to show that the
estimate provided by Corollary~GS is optimal for $k=5$.

\begin{theorem}\label{main} For every $n\in\N$, there exists a
$\K$-algebra $R$ given by $n$ generators and
$\bigl\lceil\frac{n^2}{3}\bigr\rceil$ homogeneous quadratic
relations such that $R$ is $5$-step nilpotent.
\end{theorem}

Since $\phi_5=\frac13$, Theorem~\ref{main} shows that the estimate
in Corollary~GS is optimal for $k=5$. It is worth mentioning, that
we take an advantage of the fact that $\phi_5$ is rational.
Unfortunately, our method, as it is, does not work in the case when
$\phi_k$ is irrational, and $\phi_2=1$, $\phi_3=\frac12$ and
$\phi_5=\frac13$ are the only rational numbers among $\phi_k$. For
instance, $\phi_4=\frac{3-\sqrt 5}{2}$ and, hence, the
optimality of Corollary~GS for $k=4$ and an arbitrary $n$ remains an
open question. In \cite{ns1} it is shown
that the estimate in Corollary~GS for $k=4$ is optimal when $n$ is a
Fibonacci number, while in \cite{na} its optimality is verified for
$n\leq 7$. Since $8$ is a Fibonacci number, the optimality for $n\leq 8$ follows.
Note also that Theorem~\ref{main} provides yet another proof of the conjecture of
Vershik \cite{vers} that a generic quadratic algebra over an infinite field with $n\geq 3$
generators and $\frac{n(n-1)}{2}$ relations is finite dimensional.

\section{Main lemma}

In this section we describe an algorithm of obtaining an infinite
collection of $k$-step nilpotent quadratic algebras from one
$k$-step nilpotent quadratic algebra. This is an enhanced version of
a similar algorithm from \cite{na}. Note that in the present form,
it still admits various generalizations.

Our data comprises $g_x,g_y,r_{xx},r_{xy},r_{yy}\in\N$, two disjoint
sets
$$
X=\{x_{j,s}:\text{$1\leq j\leq g_x$, $s\in\N$}\},\qquad
Y=\{y_{k,t}:\text{$1\leq k\leq g_y$, $t\in\N$}\}
$$
of pairwise distinct variables (=generators) $x_{j,s}$ and
$y_{k,t}$ and the finite collection of scalars
$$
\text{$c^{xx}_{p,j,l},c^{yy}_{q,k,w},c^{xy}_{r,j,k},c^{yx}_{r,k,j}\in\K$
\qquad for}\qquad \begin{array}{l}\text{$1\leq j,l\leq g_x$, \
$1\leq k,w\leq g_y$ and}\\ \text{$1\leq p\leq r_{xx}$, \ $1\leq
q\leq r_{yy}$, \ $1\leq r\leq r_{xy}$}.\end{array}
$$
For $\alpha,\beta\in\Z_+$, let
$$
X_\alpha=\{x_{j,s}:\text{$1\leq j\leq g_x$, $1\leq s\leq
\alpha$}\},\qquad Y_\beta=\{y_{k,t}:\text{$1\leq k\leq g_y$, $1\leq
t\leq\beta$}\}.
$$
Note that $X_0=Y_0=\varnothing$. If $\alpha+\beta\geq 1$, by $R^{(\alpha,\beta)}$ we denote the algebra given by the
generator set $X_\alpha\cup Y_\beta$ and the quadratic relations
\begin{alignat}{2}
f^{xx}_{p,a,b}&=\sum_{1\leq j,l\leq  g_x}
c^{xx}_{p,j,l}x_{j,a}x_{l,b}\quad &&\text{for $1\leq p\leq r_{xx}$,
$1\leq a,b\leq \alpha$}, \label{rab1}
\\
f^{yy}_{q,s,t}&=\sum_{1\leq k,w\leq  g_y}c^{yy}_{q,k,w}
y_{k,s}y_{w,t}\quad &&\text{for $1\leq q\leq r_{yy}$, $1\leq s,t\leq
\beta$}, \label{rab2}
\\
f^{xy}_{r,a,s}&=\sum\limits_{1\leq j\leq  g_x\atop 1\leq k\leq g_y}
(c^{xy}_{r,j,k} x_{j,a}y_{k,s}+c^{yx}_{r,k,j} y_{k,s}x_{j,a}) \quad
&&\text{for $1\leq r\leq r_{xy}$, $1\leq a\leq \alpha$, $1\leq s\leq
\beta$}. \label{rab3}
\end{alignat}

We need some extra notation. First, we denote the ideal in
$\K\langle X_\alpha\cup Y_\beta\rangle$ generated by the above
relations by the symbol $I_{\alpha,\beta}$. For a monomial (=a word)
$U$ in the alphabet $X\cup Y$, its degree (=the number of letters)
will be denoted $\deg U$. Consider the algebra homomorphisms
$$
{\cal S}_x:\K\langle X\cup Y\rangle\to \K\langle X\rangle,\ \ {\cal
S}_y:\K\langle X\cup Y\rangle\to \K\langle Y\rangle\ \ \text{and}\ \
\Phi:\K\langle X\cup Y\rangle\to \K\langle X_1\cup Y_1\rangle
$$
defined by their action on the generators as follows:
$S_x(x_{j,s})=x_{j,s}$, $S_x(y_{k,t})=1$, $S_y(x_{j,s})=1$,
$S_y(y_{k,t})=y_{k,t}$, $\Phi(x_{j,s})=x_{j,1}$ and
$\Phi(y_{k,t})=y_{k,1}$ for $1\leq j\leq g_x$, $1\leq k\leq g_y$ and
$s,t\in\N$. Informally speaking, ${\cal S}_x$ removes all
$y_{k,t}$ from each monomial, while ${\cal S}_y$ removes all
$x_{j,s}$ from each monomial. Obviously, for every $U\in\langle X\cup Y\rangle$,
\begin{align*}
\deg U&=\deg {\cal S}_x(U)+\deg {\cal S}_y(U),
\\
\deg\Phi (U)&=\deg U,
\\
\Phi({\cal S}_x (U))&={\cal S}_x(\Phi (U)),
\\
\Phi({\cal S}_y (U))&={\cal S}_y(\Phi(U)).
\end{align*}

If $n\in\N$, $n_x\in\Z$, $0\leq n_x\leq n$ and ${\bf
s}=(s_1,\dots,s_{n_x})\in\N^{n_x}$, ${\bf
t}=(t_1,\dots,t_{n_y})\in\N^{n-n_x}$, then we
define $M_{n,n_x}^{{\bf s},{\bf t}}$ to be the linear span of all
$U\in \langle X\cup Y\rangle$ for which ${\cal
S}_x(U)=x_{j_1,s_1}\dots x_{j_{n_x},s_{n_x}}$ and ${\cal
S}_y(U)=y_{k_1,t_1}\dots y_{k_{n_y},t_{n_y}}$ with $1\leq j_a\leq
g_x$ and $1\leq k_b\leq g_y$. It is straightforward to see that the
space $H_n$ of all homogeneous degree $n$ elements of $\K\langle
X\cup Y\rangle$ is the {\bf direct} sum of the spaces
$M_{n,n_x}^{{\bf s},{\bf t}}$ for all admissible choices of $n_x$,
$\bf s$ and $\bf t$. Furthermore,
\begin{equation}\label{isom}
\text{$\Phi$ restricted to $M_{n,n_x}^{{\bf s},{\bf t}}$ is a linear
isomorphism of $M_{n,n_x}^{{\bf s},{\bf t}}$ and $M_{n,n_x}^{{\bf
1},{\bf 1}}$, where ${\bf 1}=(1,\dots,1)$.}
\end{equation}

Next, if $f$ is one of the relations defined in
(\ref{rab1}--\ref{rab3}) and $U,V\in\langle X\cup Y\rangle$, then
\begin{equation}\label{TT}
\text {$UfV$ belongs to exactly one of the spaces $M_{n,n_x}^{{\bf
s},{\bf t}}$ (with $n=\deg U+\deg V+2$).}
\end{equation}
This easily follows from the shape of the relations $f$.

\begin{lemma}\label{mut}
Let $U,V\in\langle X_1\cup Y_1\rangle$ and $f$ be one of the relations defined in {\rm
(\ref{rab1}--\ref{rab3})} with $\alpha=\beta=1$. Let also $n\in\N$ and $n_x\in
\{0,\dots,n\}$ be such that $UfV\in M_{n,n_x}^{{\bf 1},{\bf 1}}$
$($they do exist and are unique according to $(\ref{TT}))$. Then for
every ${\bf s}\in\N^{n_x}$ and ${\bf t}\in\N^{n-n_x}$ with
components bounded above by $\alpha\in\N$ and $\beta\in\N$ respectively, there
are unique $U',V'\in\langle X_\alpha\cup Y_\beta\rangle$ and a unique relation
$f'$ from the list {\rm (\ref{rab1}--\ref{rab3})} such that $\Phi(U'f'V')=UfV$ and
$U'f'V'\in M_{n,n_x}^{{\bf s},{\bf t}}$.
\end{lemma}

\begin{proof} We can write
\begin{align*}
{\cal S}_x(U)&=x_{j_1,1}x_{j_2,1}\dots x_{j_a,1},
\\
{\cal S}_y(U)&=y_{k_1,1}y_{k_2,1}\dots y_{k_b,1},
\\
{\cal S}_x(V)&=x_{j_c,1}x_{j_{c+1},1}\dots x_{j_{n_x},1},
\\
{\cal S}_y(V)&=y_{k_d,1}y_{k_{d+1},1}\dots y_{k_{n-n_x},1}.
\end{align*}
Clearly, we can choose the (unique) $U'\in\langle X_\alpha\cup Y_\beta\rangle$
\begin{align*}
\Phi(U')&=U,
\\
{\cal S}_x(U')&=x_{j_1,s_1}x_{j_2,s_2}\dots x_{j_a,s_a},
\\
{\cal S}_y(U')&=y_{k_1,t_1}x_{k_2,t_2}\dots x_{y_b,t_b}.
\end{align*}
Clearly, $\deg U'=a+b$. Similarly, we can choose the (unique) $V'\in \langle X_\alpha\cup Y_\beta\rangle$
such that
\begin{align*}
\Phi(V')&=V,
\\
{\cal S}_x(V')&=x_{j_c,s_c}x_{j_{c+1},s_{c+1}}\dots x_{j_{n_x},s_{n_x}},
\\
{\cal S}_y(V)&=y_{k_d,t_d}y_{k_{d+1},t_{d+1}}\dots
y_{k_{n-n_x},t_{n-n_x}}.
\end{align*}
Obviously, $\deg V'=n-a-b-2$. If $f=f^{xx}_{p,1,1}$, we have $c=a+3$ and $d=b+1$ and
we set $f'=f^{xx}_{p,s_{a+1},s_{a+2}}$. If $f=f^{yy}_{q,1,1}$, we have
$c=a+1$ and $d=b+3$ and we set $f'=f^{yy}_{q,t_{b+1},t_{b+2}}$.
Finally, if $f=f^{xy}_{r,1,1}$, we have $c=a+2$, $d=b+2$ and we set
$f'=f^{xy}_{r,s_{a+1},s_{b+1}}$. In any case, $U'f'V'\in M_{n,n_x}^{{\bf s},{\bf t}}$
and $\Phi(U'f'V')=UfV$.
\end{proof}

\begin{lemma}\label{monid} Let $\alpha,\beta\in\Z_+$,
and $W\in\langle X_\alpha\cup Y_\beta\rangle$. Then
$W\in I_{\alpha,\beta}$ if and only if $\Phi(W)\in I_{1,1}$.
\end{lemma}

\begin{proof} It is easy to see that
$\Phi(f^{xx}_{p,a,b})=f^{xx}_{p,1,1}$,
$\Phi(f^{yy}_{q,s,t})=f^{yy}_{q,1,1}$ and
$\Phi(f^{xy}_{r,a,s})=f^{xy}_{r,1,1}$ for all possible values of
$p$, $q$, $r$, $a$, $b$, $s$ and $t$, where the relations $f$ are
defined in (\ref{rab1}--\ref{rab3}). It follows that
$\Phi(I_{\alpha,\beta})=I_{1,1}$. Thus $\Phi(W)\in I_{1,1}$ if $W\in
I_{\alpha,\beta}$. It remains to show the converse. Assume that
$\Phi(W)\in I_{1,1}$. Since $W\in\langle X_\alpha\cup
Y_\beta\rangle$, $W$ belongs to (exactly) one of the spaces
$M_{n,n_x}^{{\bf s},{\bf t}}$. Note that $n=\deg W$, $n_x=\deg{\cal
S}_x(W)$, the components of ${\bf s}$ belong to $\{1,\dots,\alpha\}$
and the components of ${\bf t}$ belong to $\{1,\dots,\beta\}$. Since
$\Phi(W)$ is a degree $n$ monomial and belongs to $I_{1,1}$, we can
write
$$
\Phi(W)=\sum_{m=1}^N \lambda_m U_m f_m V_m,
$$
where $\lambda_m\in\K\setminus\{0\}$, $U_m,V_m\in \langle X_1\cup Y_1\rangle$
and $f_1,\dots,f_N$ is a string of
relations taken from (\ref{rab1}--\ref{rab3}) with $\alpha=\beta=1$
(repetitions are allowed).

By Lemma~\ref{mut}, there are $U_m',V_m'\in\langle X_\alpha\cup Y_\beta\rangle$
and relations $f_m'$ from the list (\ref{rab1}--\ref{rab3}) such that $\Phi(U_m'f_m'V_m')=U_mf_mV_m$
and $U_m'f_m'V_m'\in M_{n,n_x}^{{\bf s},{\bf t}}$ for each $m$. That
is, the above display yields
$$
\Phi\biggl(W-\sum_{m=1}^N \lambda_m U'_m f'_m V'_m\biggr)=0.
$$
Since $W\in M_{n,n_x}^{{\bf s},{\bf t}}$ and each $U'_m f'_m V'_m$ belongs to $M_{n,n_x}^{{\bf
s},{\bf t}}$, (\ref{isom}) implies that
$$
W=\sum_{m=1}^N \lambda_m U'_m f'_m V'_m.
$$
Since each $f'_m$ belongs to the ideal $I_{\alpha,\beta}$, we have $W\in
I_{\alpha,\beta}$.
\end{proof}

The following lemma is our key instrument.

\begin{lemma}\label{key} If $R^{(1,1)}$ is $n$-step nilpotent, then
$R^{(\alpha,\beta)}$ is $n$-step nilpotent for every
$\alpha,\beta\in\Z_+$.
\end{lemma}

\begin{proof} Without loss of generality we can assume that $n\geq 3$. Assume
that $R^{(1,1)}$ is $n$-step nilpotent. Let $W$ be a degree $n$
monomial in the alphabet $X_\alpha\cup Y_\beta$. The proof will be
complete if we verify that $W$ belongs to the ideal
$I_{\alpha,\beta}$ generated by the relations
(\ref{rab1}--\ref{rab3}). Clearly, $\Phi(W)$ is a degree $n$
monomial in the alphabet $X_1\cup Y_1$. Since $R^{(1,1)}$ is
$n$-step nilpotent $\Phi(W)\in I_{1,1}$. By Lemma~\ref{monid}, $W\in
I_{\alpha,\beta}$, which completes the proof.
\end{proof}

Note that $R^{(\alpha,\beta)}$ has exactly $\alpha g_x+\beta g_y$
generators and $\alpha^2r_{xx}+\beta^2r_{yy}+\alpha\beta r_{xy}$
relations. This observation and the shape of the relations in
$R^{(1,1)}$ allow us to give the following reformulation of
Lemma~\ref{key}. As usual, for subsets $A$ and $B$ of an algebra
$R$, $A\cdot B$ stands for $\spann\{ab:a\in A,\,b\in B\}$.

\begin{corollary}\label{gjsdfg} Assume that there exists a quadratic
algebra $R$ given by the $(g_x+g_y)$-element set $X\cup Y$ of
generators with $|X|=g_x$ and $|Y|=g_y$, $r_{xx}$ relations
belonging to $X\cdot X$, $r_{yy}$ relations belonging to $Y\cdot Y$
and $r_{xy}$ relations belonging to $X\cdot Y+Y\cdot X$ such that
$R$ is $n$-step nilpotent. Then for every $\alpha,\beta\in\Z_+$,
there exists a quadratic $\K$-algebra $R^{(\alpha,\beta)}$ given by
$\alpha g_x+\beta g_y$ generators and
$\alpha^2r_{xx}+\beta^2r_{yy}+\alpha\beta r_{xy}$ relations such
that $R^{(\alpha,\beta)}$ is $n$-step nilpotent.
\end{corollary}

\section{Two specific algebras}

\begin{lemma}\label{3-1} Let $\K$ be an arbitrary field and $R_{3,1}$ be
the $\K$-algebra given by $4$ generators $a,b,c,x$ and $6$ relations
$cb-bc+aa$, $bb+aa-ac$, $cc-ba$, $xx$, $ax+bx-xa-xb$ and
$bx+cx-xa-xb-xc$. Then $R_{3,1}$ is $5$-step nilpotent.
\end{lemma}

\begin{lemma}\label{3-2} Let $\K$ be an arbitrary field and $R_{3,2}$ be
the $\K$-algebra given by $5$ generators $a,b,c,x,y$ and $9$
relations $cb-bc+aa$, $bb+aa-ac$, $cc-ba$, $yy-xx$, $yx$,
$ay+bx+cy-xb-yc$, $ax+by-xa-yb$, $cy+bx-ya-xb-yc$ and
$ax+by+cx-ya-xb$. Then $R_{3,2}$ is $5$-step nilpotent.
\end{lemma}

Our proof of Lemmas~\ref{3-1} and~\ref{3-2} is computer assisted. We order the
generators by $a<b<c<x$ for $R_{3,1}$ and by $a<b<c<x<y$ for $R_{3,2}$ and use
the degree lexicographical ordering on the monomials. We have used
the program Bergman \cite{berg} developed by J\"orgen Backelin
(Stockholm University). The software allows to obtain the Gr\"obner
basis and the Hilbert series of an associative algebra given by
generators and homogeneous relations. We list the (finite) Gr\"obner bases of
$R_{3,1}$ and of $R_{3,2}$ in the case char$\,\K=0$ in the appendix. Note that the
elements of the Gr\"obner bases are given as members of $\Z\langle a,b,c,x\rangle$ for $R_{3,1}$
and as members of $\Z\langle a,b,c,x,y\rangle$ for $R_{3,2}$. For each of the algebras $R_{3,1}$
and $R_{3,2}$ we have run the program twice specifying ${\rm char}\,\K=0$ for the first run
and ${\rm char}\,\K=2$ for the second. The reason for this will be apparent straight away.

Right now we point out only the relevant consequences of the output:
\begin{align}\label{fof1}
&\begin{array}{l}\text{The Hilbert series of $R_{3,1}$ and $R_{3,2}$ in
the case ${\rm char}\,\K\in\{0,2\}$ are}\\
\text{$H_{R_{3,1}}(t)=1+4t+10t^2+18t^3+21t^4$ and
$H_{R_{3,2}}(t)=1+5t+16t^2+35t^3+43t^4$}.\end{array}
\\
\label{fof2}
&\begin{array}{l}\text{In the case ${\rm char}\,\K=0$, the leading coefficients of all the members} \\
\text{of the Gr\"obner bases of $R_{3,1}$ and $R_{3,2}$ are of the form $\pm 2^j$ for $j\in\Z_+$.}\end{array}
\end{align}

{\bf Now we are ready to prove Lemmas~\ref{3-1} and~\ref{3-2}.} \ Let $p$ be an odd prime number.
By (\ref{fof2}) none of the leading coefficients of the members $g_j$ of the
characteristic 0 Gr\"obner bases of $R_{3,1}$ and $R_{3,2}$ is a multiple of $p$. It follows that if we
replace all the coefficients $c$ of each $g_j\in \Z\langle a,b,c,x,y\rangle$ by the corresponding cosets
$c+p\Z\in \Z/p\Z$, we obtain Gr\"obner bases of $R_{3,1}$ and $R_{3,2}$ in the case ${\rm char}\,\K=p$ and the
leading monomials are the same as in the case ${\rm char}\,\K=0$. Since the Hilbert series is uniquely
determined by the set of leading monomials of the Gr\"obner basis, the Hilbert series of $R_{3,1}$ and $R_{3,2}$
in the case ${\rm char}\,\K=p$ are the same as in the case ${\rm char}\,\K=0$. Looking at (\ref{fof1}),
we see that
$$
\begin{array}{l}\text{For every field $\K$,
the Hilbert series of $R_{3,1}$ and $R_{3,2}$ are}\\
\text{$H_{R_{3,1}}(t)=1+4t+10t^2+18t^3+21t^4$ and
$H_{R_{3,2}}(t)=1+5t+16t^2+35t^3+43t^4$}.\end{array}
$$
It immediately follows that $R_{3,1}$ and $R_{3,2}$ are 5-step nilpotent, which completes the proof of
Lemmas~\ref{3-1} and~\ref{3-2}.

\section{Proof of Theorem~\ref{main}}

In this section we derive Theorem~\ref{main} from
Corollary~\ref{gjsdfg} and Lemmas~\ref{3-1} and~\ref{3-2}.

By Lemma~\ref{3-1}, there exists a quadratic algebra $R_{3,1}$ given
by the set $X\cup Y$ of generators with $X=\{a,b,c\}$ and $Y=\{x\}$,
$3$ relations ($cb-bc+aa$, $bb+aa-ac$ and $cc-ba$) from $X\cdot X$,
$1$ relation (being $xx$) from $Y\cdot Y$ and $2$ relations
($ax+bx-xa-xb$ and $bx+cx-xa-xb-xc$) from $X\cdot Y+Y\cdot X$ such
that $R_{3,1}$ is $5$-step nilpotent. By Corollary~\ref{gjsdfg} with
$\alpha\in\N$ and $\beta=0$, we see that
\begin{equation}\label{3+0}
\begin{array}{l}
\text{for every $\alpha\in\N$, there is a quadratic $\K$-algebra
$R_{3\alpha}$ given by $3\alpha$ generators}
\\
\text{and $3\alpha^2=\bigl\lceil\frac{(3\alpha)^2}{3}\bigr\rceil$
relations such that $R_{3\alpha}$ is $5$-step nilpotent.}
\end{array}
\end{equation}
By Corollary~\ref{gjsdfg} applied to the same algebra $R_{3,1}$ with $\alpha\in\Z_+$ and $\beta=1$,
we see that
\begin{equation}\label{3+1}
\begin{array}{l}
\text{for every $\alpha\in\Z_+$, there is a quadratic $\K$-algebra
$R_{3\alpha+1}$ given by $3\alpha+1$ generators}
\\
\text{and
$3\alpha^2+2\alpha+1=\bigl\lceil\frac{(3\alpha+1)^2}{3}\bigr\rceil$
relations such that $R_{3\alpha+1}$ is $5$-step nilpotent.}
\end{array}
\end{equation}

By Lemma~\ref{3-2}, there exists a quadratic algebra $R_{3,2}$ given
by the set $X\cup Y$ of generators with $X=\{a,b,c\}$ and
$Y=\{x,y\}$, $3$ relations ($cb-bc+aa$, $bb+aa-ac$ and $cc-ba$) from
$X\cdot X$, $2$ relation ($yy-xx$ and $yx$) from $Y\cdot Y$ and $4$
relations ($ay+bx+cy-xb-yc$, $ax+by-xa-yb$, $cy+bx-ya-xb-yc$ and
$ax+by+cx-ya-xb$) from $X\cdot Y+Y\cdot X$ such that $R$ is $5$-step
nilpotent. By Corollary~\ref{gjsdfg} with $\alpha\in\Z_+$ and
$\beta=1$, we see that
\begin{equation}\label{3+2}
\begin{array}{l}
\text{for every $\alpha\in\Z_+$, there is a quadratic $\K$-algebra
$R_{3\alpha+2}$ given by $3\alpha+2$ generators}
\\
\text{and
$3\alpha^2+4\alpha+2=\bigl\lceil\frac{(3\alpha+2)^2}{3}\bigr\rceil$
relations such that $R_{3\alpha+2}$ is $5$-step nilpotent.}
\end{array}
\end{equation}

The formulae (\ref{3+0}--\ref{3+2}) show that for every $n\in\N$,
there is a quadratic $\K$-algebra $R_{n}$ given by $n$ generators
and $\bigl\lceil\frac{n^2}{3}\bigr\rceil$ relations such that $R_n$
is $5$-step nilpotent. This completes the proof of
Theorem~\ref{main}.

\section{Concluding remarks and open questions}

We start by reiterating the message that the optimality of Corollary~GS for $k\notin\{2,3,5\}$ remains an open question.

\begin{question}\label{q1} Let $k$ be an integer such that $k=4$ or $k\geq 6$. Is it true that for each $n\in\N$, there
is a $k$-step nilpotent quadratic algebra given by $n$ generators and $\lceil\phi_kn^2\rceil$ relations$?$
\end{question}

Another consequence of Theorem~GS is the following corollary.

\begin{corgs1} A quadratic algebra given by $n$ generators and $d\leq\frac{n^2}{4}$ relations is infinite
dimensional.
\end{corgs1}

The question of the optimality of Corollary~GS1 remains open as well.

\begin{question}\label{q2} Let $n\in\N$ and $d$ be an integer strictly greater than $\frac{n^2}{4}$.
Is it true that there is a finite dimensional quadratic algebra given by $n$ generators and $d$ relations$?$
\end{question}

Note that Wisliceny \cite{wis} proved that Corollary~GS1 is asymptotically optimal.
Namely, he proved that for every $n\in\N$ there is a semigroup quadratic algebra (every relation is either a monomial
of degree 2 or a difference of 2 monomials of degree 2) given by $n$ generators and $\bigl\lceil\frac{n^2+2n}{4}\bigr\rceil$
relations. The authors \cite{ns} proved that in the class of semigroup quadratic algebras with $n$ generators, the minimal
number of relations required for the algebra to be finite dimensional is exactly the first integer strictly greater than
$\frac{n^2+n}{4}$. In particular, this roughly halves the gap between the $\frac{n^2}{4}$ of Corollary~GS1 and
about $\frac{n^2+2n}{4}$ of Wisliceny. As it is observed in \cite{ns}, for $k\geq 5$, the number of semigroup quadratic
relations on $n\geq 3$ generators required for $k$-step nilpotency is strictly greater than the number of general quadratic
relations. Due to Anick \cite{ani1}, these numbers are the same for $k=3$. The case $k=4$ remains a mystery, which leads to
the following question.

\begin{question}\label{q3} Let $n\in\N$. Is it true that there is a $4$-step nilpotent
semigroup quadratic algebra given by $n$ generators and $\lceil\phi_4n^2\rceil$ relations$?$
\end{question}

To make the above question more appealing, we make the following observation.

\begin{proposition}\label{4st} The answer to Question~$\ref{q3}$ is affirmative for $n\leq 5$.
\end{proposition}

\begin{proof} Note that for $1\leq n\leq 5$, the numbers $\lceil\phi_4n^2\rceil$ are $1$, $2$, $4$, $7$ and $10$.
Consider the algebras $R_n$ for $1\leq n\leq 5$ defined by
\begin{align*}
R_1&=\K\langle a\rangle/{\rm Id}(a^2);
\\
R_2&=\K\langle a,b\rangle/{\rm Id}(b^2{-}a^2,ba);
\\
R_3&=\K\langle a,b,c\rangle/{\rm Id}(c^2{-}ba, cb{-}a^2,b^2,ca);
\\
R_4&=\K\langle a,b,c,d\rangle/{\rm Id}(d^2{-}ca,dc{-}ab,db{-}a^2,da,cd{-}b^2,c^2{-}ba,cb{-}bc);
\\
R_5&=\K\langle a,b,c,d,e\rangle/{\rm Id}(de{-}eb,ce{-}db,e^2{-}da,ed{-}cb,d^2{-}b^2,cd{-}ab,ec{-}ca,dc{-}ba,c^2{-}a^2,ea).
\end{align*}
Then each $R_n$ for $1\leq n\leq 5$ is given by $n$ generators and $\lceil\phi_4n^2\rceil$ semigroup quadratic relations.
The proof will be complete if we show that each $R_n$ is $4$-step nilpotent. The fact that $R_1$, $R_2$ and $R_3$
are 4-step nilpotent is easily verifiable by hand  (in fact $R_1$ is 2-step nilpotent and $R_2$ is 3-step nilpotent).
We have again used BERGMAN to verify that $R_4$ and $R_5$ are 4-step nilpotent. Since for a semigroup algebra the Hilbert series does not depend on the underlying field,
we can run the program to obtain the Hilbert series of $R_n$ choosing an arbitrary characteristic of $\K$ (we have opted for
char$\,\K=0$). This gives the Hilbert series of $R_n$:
\begin{align*}
H_{R_1}(t)&=1+t;
\\
H_{R_2}(t)&=1+2t+2t^2;
\\
H_{R_3}(t)&=1+3t+5t^2+4t^3;
\\
H_{R_4}(t)&=1+4t+9t^2+8t^3;
\\
H_{R_5}(t)&=1+5t+15t^2+25t^3.
\end{align*}
Thus each $R_n$ for $1\leq n\leq 5$ is $4$-step nilpotent.
\end{proof}

We have to confess that the algebras $R_4$  and $R_5$ in the above proof are obtained via computer assisted guesswork.

\section{Acknowledgements}

   This research is funded by the ERC grant 320974. We also would like to acknowledge hospitality and support of IHES and MPIM, Bonn, where part of this work has been done. Earlier stages of this research was supported by the project PUT9038.


\small\rm

\section{Appendix: The Gr\"obner bases of $R_{3,1}$ and $R_{3,2}$}

{\bf Table 1.} \sl If ${\rm char}\,\K=0$, the Gr\"obner basis of $R_{3,1}$ is \rm
\small
\begin{align*}
&g_1=b^2{-}ac{+}a^2;\ \ g_2=cb{-}bc{+}a^2;\ \ g_3=c^2{-}ba;\ \ g_4={-}xb{-}xa{+}bx{+}ax;\ \ g_5={-}xc{+}cx{-}ax;
\\
&g_6=x^2;\ \ g_7=bac{-}ba^2{-}abc{+}a^2b{+}a^3;\ \ g_8=bca{-}ba^2{-}abc{+}a^2b;\ \ g_9={-}ca^2{-}bab{+}aca{-}a^2c{-}a^3;
\\
&g_{10}=cac{+}bab{+}ba^2{-}aca{-}aba{+}2a^2c{+}a^2b{+}a^3;\ \ g_{11}=xa^2{-}cax{-}bxa{+}bax{-}axa{-}acx{+}a^2x;
\\
&g_{12}={-}xab{-}cxa{+}3cax{+}bxa{+}2axa{+}2acx{-}3a^2x;\ \ g_{13}={-}xac{+}cxa{-}2cax{-}bxa{-}2axa{+}abx{+}2a^2x;\ \ g_{14}=xax;
\\
&g_{15}=ba^3{-}abab{-}2aba^2{+}2a^2ca{+}a^2ba{-}3a^3c{-}a^3b{-}a^4;
\\
&g_{16}={-}ba^2b{+}acab{+}3abab{+}5aba^2{-}5a^2ca{+}a^2bc{-}3a^2ba{+}6a^3c{+}5a^4;
\\
&g_{17}={-}ba^2c{-}acab{-}abab{-}2aba^2{+}2a^2ca{+}a^2ba{-}3a^3c{-}2a^4;
\\
&g_{18}=baba{+}acab{+}abab{+}2aba^2{-}3a^2ca{+}a^2bc{-}a^2ba{+}4a^3c{+}3a^4;
\\
&g_{19}=babc{-}acab{-}4abab{-}6aba^2{+}7a^2ca{-}a^2bc{+}3a^2ba{-}7a^3c{-}7a^4;
\\
&g_{20}={-}{\bf8}bcxa{+}15baxa{-}20babx{+}ba^2x{-}5acxa{+}32acax{+}4abxa{+}4abcx{-}29abax{+}8a^2xa{+}24a^2cx{+}5a^2bx{-}31a^3x;
\\
&g_{21}=caba{+}4abab{+}4aba^2{-}6a^2ca{+}a^2bc{-}a^2ba{+}7a^3c{+}2a^3b{+}7a^4;\ \ g_{22}=cabc{-}2acab{-}aba^2{+}2a^2ca{+}2a^2bc{-}2a^3c{-}2a^4;
\\
&g_{23}={\bf 8}cabx{-}3baxa{+}12babx{+}19ba^2x{+}17acxa{+}8acax{+}12abxa{+}44abcx{+}9abax{-}49a^2bx
\\
&\quad\qquad{-}
37a^3x{-}8caxa{+}baxa{-}4babx{-}9ba^2x{-}11acxa{-}4abxa{-}12abcx{+}5abax{+}8a^2xa{-}8a^2cx{+}11a^2bx{+}7a^3x;
\\
&g_{24}=a^5;\ \ g_{25}=a^4b;\ \ g_{26}=a^4c;\ \ g_{27}=a^4x;\ \ g_{28}=a^3ba;\ \ g_{29}=a^3bc;\ \ g_{30}=a^3bx;\ \ g_{31}=a^3ca;
\\
&g_{32}=a^3cx;\ \ g_{33}=a^3xa;\ \ g_{34}=a^2ba^2;\ \ g_{35}=a^2bab;\ \ g_{36}=a^2bax;\ \ g_{37}=a^2bcx;\ \ g_{38}=a^2bxa;\ \ g_{39}=a^2cab;
\\
&g_{40}=a^2cax;\ \ g_{41}=a^2cxa;\ \ g_{42}=aba^2x;\ \ g_{43}=ababx;\ \ g_{44}=abaxa;\ \ g_{45}=ba^2xa;\ \ g_{46}=babxa.
\end{align*}

\vskip1truecm

{\bf Table 2.} \sl If ${\rm char}\,\K=0$, the Gr\"obner basis of
$R_{3,2}$ is \rm
\small
\begin{align*}
&g_1=b^2{-}ac{+}a^2;\ \ g_2=cb{-}bc{+}a^2;\ \ g_3=c^2{-}ba;\ \ g_4=xb{-}cx{-}by{-}ay{-}ax;\ \ g_5=-ya-ay;
\\
&g_6={-}yb{-}xa{+}by{+}ax;\ \ g_7={-}yc{+}cy{-}cx{-}by{+}bx{-}ax;\ \ g_8=yx;\ \ g_9=y^2{-}x^2;\ \
g_{10}=bac{-}ba^2{-}abc{+}a^2b{+}a^3;
\\
&g_{11}=bca{-}ba^2{-}abc{+}a^2b;\ \ g_{12}=ca^2{-}bab{+}aca{-}a^2c{-}a^3;\ \
g_{13}=cac{+}bab{+}ba^2{-}aca{-}aba{+}2a^2c{+}a^2b{+}a^3;
\\
&g_{14}={-}cxc{+}cxa{-}2cax{+}bxc{-}2bcy{+}bay{-}2bax{-}acy{-}aby{+}2a^2y{-}2a^2x;
\\
&g_{15}=xa^2{-}cxa{+}cax{+}2bcx{+}bay{+}2bax{-}axc{-}axa{+}2acy{-}acx{+}aby{-}2a^2y{+}2a^2x;
\\
&g_{16}={-}xab{-}bxa{+}bax{+}2acy{+}abx{+}a^2y;
\\
&g_{17}=xac{-}cxa{-}cay{+}bxa{-}bcy{+}2bcx{+}bay{-}axc{+}acy{-}2acx{-}aby{-}a^2y;
\\
&g_{18}=-xay-cx^2{+}2bx^2{+}axy-2ax^2;\ \ g_{19}={-}xcx{-}xax{-}cxy{+}2cx^2{-}4bx^2{-}2axy{+}3ax^2;
\\
&g_{20}={-}x^2a{+}ax^2;\ \ g_{21}={-}x^2c{-}xax{-}cxy{+}cx^2{+}bxy{-}bx^2{-}axy;\ \ g_{22}=x^3;
\\
&g_{23}=x^2y;\ \ g_{24}=ba^3{-}abab{-}2aba^2{+}2a^2ca{+}a^2ba{-}3a^3c{-}a^3b{-}a^4;
\\
&g_{25}={-}ba^2b{+}acab{+}3abab{+}5aba^2{-}5a^2ca{+}a^2bc{-}3a^2ba{+}6a^3c{+}5a^4;
\\
&g_{26}={-}ba^2c{-}acab{-}abab{-}2aba^2{+}2a^2ca{+}a^2ba{-}3a^3c{-}2a^4;
\\
&g_{27}=baba{+}acab{+}abab{+}2aba^2{-}3a^2ca{+}a^2bc{-}a^2ba{+}4a^3c{+}3a^4;
\\
&g_{28}=babc{-}acab{-}4abab{-}6aba^2{+}7a^2ca{-}a^2bc{+}3a^2ba{-}7a^3c{-}7a^4;
\\&g_{29}={-}baxc{+}baxa{-}baby{+}ba^2y{-}4ba^2x{+}2acxa{+}2acay{-}2acax{-}ab
     xa{-}abcy{-}4abcx
\\
&\qquad\quad {-}3abay{+}2abax{-}a^2xc{+}a^2xa{+}a^2cy{+}3a^2cx{+}2a^
     2by{+}a^2bx{-}a^3y{+}3a^3x;
\\&g_{30}={-}bcxa{-}3baxa{-}2baby{-}4ba^2y{-}3ba^2x{+}axca{+}3acxa{+}6acay{-}
     2acax{-}abxc{-}3abcy{-}5abcx
\\
&\qquad\quad {-}abay{+}2abax{+}a^2xc{-}3a^2xa{-}3a^2
     cy{+}5a^2cx{+}5a^2by{+}a^2bx{-}5a^3y{+}5a^3x;
\\
&g_{31}={-}{\bf 2}bcx^2{+}2baxy{-}3bax^2{-}4axcy{+}7axax{+}7acxy{-}12acx^2{-}13a
     bxy{+}7abx^2{+}9a^2xy{-}23a^2x^2;
\\
&g_{32}={\bf 2}bcxy{+}3baxy{-}3axcy{+}2axax{+}acxy{-}7acx^2{-}5abxy{+}6abx^2
     {+}5a^2xy{-}17a^2x^2;
\\
&g_{33}={\bf 2}bxax{-}23baxy{+}14bax^2{+}31axcy{-}48axax{-}51acxy{+}117acx^2{+}
     93abxy{-}86abx^2{-}71a^2xy{+}227a^2x^2;
\\
&g_{34}={-}bxca{+}baxa{-}baby{+}3ba^2y{-}3acxa{-}6acay{-}2acax{+}2abxc{+}2
     abxa{-}abcy{+}6abcx{+}4abay
\\
&\qquad\quad  {-}10abax{+}a^2xa{-}2a^2cy{-}9a^2cx{-}5a^2
     by{-}a^2bx{+}11a^3y{-}7a^3x;
\\
&g_{35}={\bf 2}bxcy{+}66baxy{-}51bax^2{-}82axcy{+}130axax{+}138acxy{-}312acx^
     2{-}250abxy{+}236abx^2{+}189a^2xy{-}637a^2x^2;
\\
&g_{36}=caba{+}4abab{+}4aba^2{-}6a^2ca{+}a^2bc{-}a^2ba{+}7a^3c{+}2a^3b{+}7a^4;
\\
&g_{37}=cabc{-}2acab{-}aba^2{+}2a^2ca{+}2a^2bc{-}2a^3c{-}2a^4;
\\
&g_{38}={-}cabx{-}7baxa{-}6baby{-}5ba^2y{-}19ba^2x{+}3axca{+}10acxa{+}19ac
     ay{-}11acax{-}3abxc{-}2abxa{-}11abcy
\\
&\qquad\quad   {-}22abcx{-}9abay{+}7abax{-}8
     a^2xa{-}2a^2cy{+}24a^2cx{+}21a^2by{+}5a^2bx{-}14a^3y{+}20a^3x;
\\
&g_{39}={-}caby{-}3baxa{-}3baby{+}2babx{-}2ba^2y{-}7ba^2x{+}axca{+}2acxa{+}
     4acay{-}6acax{+}abxa{-}4abcy{-}5abcx
\\
&\qquad\quad      {-}2abax{-}4a^2xa{-}a^2cy{+}4a^
     2cx{+}3a^2by{+}4a^2bx{-}2a^3y{+}6a^3x;
\\
&g_{40}=
     caxa{+}baxa{+}3baby{+}6ba^2x{-}axca{-}acxa{-}3acay{+}4acax{+}3ab
     cy{+}3abcx
 \\
&\qquad\quad      {+}abax{+}a^2xc{+}3a^2xa{-}4a^2cx{-}4a^2by{-}a^2bx{+}a^3y{-}4a^3x;
\\
&g_{41}={-}caxc{-}3baxa{+}babx{-}4ba^2y{+}7ba^2x{-}3acxa{-}4acay{+}5acax{+}
     3abxa{+}2abcy{+}7abcx{+}6abay
\\
&\qquad\quad {-}3abax{+}a^2xc{-}2a^2xa{-}3a^2cy{-}9
     a^2cx{-}7a^2by{-}a^3y{-}5a^3x;
\\
&g_{42}={-}cax^2{+}baxy{-}bax^2{-}2axcy{+}3axax{+}3acxy{-}7acx^2{-}6abxy{+}4
     abx^2{+}4a^2xy{-}14a^2x^2;
\\
&g_{43}={-}{\bf 2}caxy{-}7baxy{+}10bax^2{+}13axcy{-}23axax{-}24acxy{+}45acx^2{+}
     42abxy{-}29abx^2{-}33a^2xy{+}75a^2x^2;
\\
&g_{44}={\bf 2}cxax{-}17baxy{+}12bax^2{+}25axcy{-}40axax{-}43acxy{+}95acx^2{+}
     77abxy{-}68abx^2{-}59a^2xy{+}179a^2x^2;
\\
&g_{45}={\bf 2}xaxa{-}11baxy{+}8bax^2{+}13axcy{-}16axax{-}17acxy{+}47acx^2{+}35
     abxy{-}34abx^2{-}25a^2xy{+}101a^2x^2;
\\
&g_{46}={-}{\bf 2}xaxc{+}9baxy{-}2bax^2{-}7axcy{+}9axax{+}10acxy{-}25acx^2{-}18a
     bxy{+}21abx^2{+}15a^2xy{-}61a^2x^2;
\end{align*}

\begin{align*}
&g_{47}=xax^2;\ \ g_{48}=xaxy;
\\
&g_{49}={\bf 2}xcax{+}39baxy{-}17bax^2{-}63axcy{+}102axax{+}105acxy{-}235acx^
     2{-}191abxy{+}150abx^2{+}140a^2xy{-}412a^2x^2;
\\
&g_{50}={-}{\bf 2}xcay{+}65baxy{-}45bax^2{-}83axcy{+}121axax{+}128acxy{-}313acx
     ^2{-}244abxy{+}225abx^2{+}182a^2xy{-}630a^2x^2;
\\
&g_{51}=a^5;\ \ g_{52}=a^4b;\ \ g_{53}=a^4c;\ \ g_{54}=a^4x;\ \  g_{55}=a^4y;\ \ g_{56}=a^3ba;\ \
g_{57}=a^3bc;
\\
&g_{58}=a^3bx;\ \ g_{59}=a^3by; \ \ g_{60}=a^3ca;\ \ g_{61}=a^3cx;\ \ g_{62}=a^3cy;\ \ g_{63}=a^3xa;\ \
g_{64}=a^3xc;
\\
&g_{65}=a^3x^2;\ \ g_{66}=a^3xy;\ \ g_{67}=a^2ba^2;\ \ g_{68}=a^2bab;\ \ g_{69}=a^2bax;\ \
g_{70}=a^2bay;\ \ g_{71}=a^2bcx;
\\
&g_{72}=a^2bcy;\ \ g_{73}=a^2bxa;\ \ g_{74}=a^2bxc;\ \ g_{75}=a^2bx^2;\ \ g_{76}=a^2bxy;\ \
g_{77}=a^2cab;\ \ g_{78}=a^2cax;
\\
&g_{79}=a^2cay;\ \ g_{80}=a^2cxa;\ \ g_{81}=a^2cx^2;\ \ g_{82}=a^2cxy;\ \ g_{83}=a^2xax;\ \
g_{84}=a^2xca;\ \ g_{85}=a^2xcy;
\\
&g_{86}=aba^2x;\ \ g_{87}=aba^2y;\ \ g_{88}=ababx;\ \ g_{89}=ababy;\ \ g_{90}=abaxa;\ \
g_{91}=abax^2;\ \ g_{92}=abaxy;
\\
&g_{93}=axcab;\ \ g_{94}=ba^2xa;\ \ g_{95}=ba^2xc;\ \ g_{96}=ba^2x^2;\ \ g_{97}=ba^2xy;\ \
g_{98}=babxa;\ \ g_{99}=babxc;
\\
&g_{100}=babx^2;\ \ g_{101}=babxy;\ \ g_{102}=baxax.
\end{align*}

\normalsize

\vskip1truecm

\scshape

\noindent  Natalia Iyudu\ \

\noindent School of Mathematics

\noindent  The University of Edinburgh

\noindent James Clerk Maxwell Building

\noindent The King's Buildings

\noindent Mayfield Road

\noindent Edinburgh

\noindent Scotland EH9 3JZ

\noindent E-mail address: \qquad {\tt niyudu@staffmail.ed.ac.uk}\ \ \

{\rm and}\ \ \

\noindent   Stanislav Shkarin

\noindent Queens's University Belfast

\noindent Pure Mathematics Research Centre

\noindent University road, Belfast, BT7 1NN, UK

\noindent E-mail address: \qquad {\tt s.shkarin@qub.ac.uk}


\begin{thebibliography}{99}

\itemsep=-2pt

\bibitem{ani1}D.~Anick, \it Generic algebras and CW complexes, \rm Algebraic
topology and algebraic $K$-theory (Princeton, N.J., 1983), 247--321,
Ann. of Math. Stud. \bf113\rm, Princeton Univ. Press, Princeton, NJ,
1987

\bibitem{ani2}D.~Anick, \it Noncommutative graded algebras and their Hilbert
series, \rm  J. Algebra \bf78\rm\ (1982), 120--140

\bibitem{berg}BERGMAN http://servus.math.su.se/bergman/
J\"orgen Backelin (Department of Mathematics, Stockholm University)


\bibitem{cana}P.~Cameron and N.~Iyudu, \it Graphs of relations and Hilbert
series, \rm J. Symbolic Comput. \bf42\rm\ (2007), 1066--1078

\bibitem{Ersh}Ershov, M., \it Golod–Shafarevich groups: A survey. \rm Int. J. Algebra Comput. {\bf 22}(2012), N5, 1–-68 

\bibitem{gosh}E.~Golod and I.~Shafarevich, \it On the class field tower \rm
(Russian), Izv. Akad. Nauk SSSR Ser. Mat. \bf28\rm\ (1964), 261--272

\bibitem{na}N.~Iyudu and S.~Shkarin, \it The Golod-Shafarevich inequality
for Hilbert series of quadratic algebras and the Anick conjecture,
\rm Proc. Roy. Soc. Edinburgh \bf A141\rm\ (2011), 609--629

\bibitem{ns}N.~Iyudu and S.~Shkarin, \it Finite dimensional semigroup quadratic
algebras with minimal number of relations, \rm Monats. Math. \bf168\rm\ (2012), 239--252

\bibitem{ns1}N.~Iyudu and S.~Shkarin, \it Asymptotically optimal $k$-step
nilpotency of quadratic algebras and the Fibonacci numbers, \rm Combinatorica,
to appear

\bibitem{popo}A.~Polishchuk and L.~Positselski, \it Quadratic
algebras, \rm University Lecture Series \bf37\rm\ American
Mathematical Society, Providence, RI, 2005

\bibitem{Sm1}Smoktunowicz, A., Bartholdi, L.\it Images of Golod–Shafarevich algebras with small growth, \rm Q. J. Math. (2013). doi:10.1093/qmath/hat005

 \bibitem{Sm2}Smoktunowicz, A.\it Golod-Shafarevich algebras, free subalgebras and homomorphic images, \rm  J. Algebra {\bf 381} (2013) 116–130
 
 \bibitem{Sm3}Smoktunowicz, A.\it Growth, entropy and commutativity of algebras satisfying prescribed relations, \rm Sel.Math.New.Ser.  DOI 10.1007/s00029-014-0154-x, (2014)




\bibitem{ufnar}V.~Ufnarovskii, \it Combinatorial and asymptotic methods in
algebra \rm (Russian) Current problems in mathematics. Fundamental
directions \bf57\rm\ 5--177, Itogi Nauki i Tekhniki, Akad. Nauk
SSSR, Moscow, 1990

\bibitem{vers}A.~Vershik, \it Algebras with quadratic relations, \rm
Selecta Math. Soviet. \bf11\rm\ (1992), 293--315

\bibitem{wis}I.~Wisliceny, \it Konstruktion nilpotenter associativer
Algebren mit wenig Relationen, \rm Math. Nachr. \bf147\rm\ (1990),
75--82

\bibitem{Zelm}Zelmanov, E. \it Some open problems in the theory of infinite dimensional algebras, \rm J. Korean Math. Soc. {\bf 44}(2007), N5, 1185–1195 

\end{thebibliography}
\end{document}